\documentclass[11pt,twoside]{amsart}
\usepackage[latin1]{inputenc}
\usepackage[OT1]{fontenc}
\usepackage{amsmath}
\usepackage{amsthm}
\usepackage{amssymb}
\usepackage[all]{xy}

\newtheorem{thm}{Theorem}[section]

\newtheorem{prop}[thm]{Proposition}

\newtheorem{lem}[thm]{Lemma}
\newtheorem{cor}[thm]{Corollary}

\numberwithin{equation}{section}

\theoremstyle{definition}
\newtheorem{definition}[thm]{Definition}
\newtheorem{remark}[thm]{Remark}
\newtheorem{ex}[thm]{Example}

\newcommand{\qqed}{\hspace*{\fill}$\Box$}

\newcommand{\Db}{{\rm D}^{\rm b}}

\newcommand{\Aut}{{\rm Aut}}

\newcommand{\Pic}{{\rm Pic}}

\newcommand{\rk}{{\rm rk}}

\newcommand{\Hom}{{\rm Hom}}

\newcommand{\Stab}{{\rm Stab}}

\newcommand{\id}{{\rm id}}

\newcommand{\ka}{{\mathcal A}}

\newcommand{\kc}{{\mathcal C}}
\newcommand{\kd}{{\mathcal D}}

\newcommand{\kf}{{\mathcal F}}

\newcommand{\ko}{{\mathcal O}}
\newcommand{\kp}{{\mathcal P}}

\newcommand{\kt}{{\mathcal T}}

\newcommand{\IC}{\mathbb{C}}

\newcommand{\IP}{\mathbb{P}}

\newcommand{\IR}{\mathbb{R}}
\newcommand{\IZ}{\mathbb{Z}}

\DeclareMathOperator{\Coh}{\bf{Coh}}

\renewcommand{\to}{\xymatrix@1@=15pt{\ar[r]&}}
\renewcommand{\rightarrow}{\xymatrix@1@=15pt{\ar[r]&}}
\renewcommand{\mapsto}{\xymatrix@1@=15pt{\ar@{|->}[r]&}}
\renewcommand{\twoheadrightarrow}{\xymatrix@1@=15pt{\ar@{->>}[r]&}}
\renewcommand{\hookrightarrow}{\xymatrix@1@=15pt{\ar@{^(->}[r]&}}
\newcommand{\congpf}{\xymatrix@1@=15pt{\ar[r]^-\sim&}}
\renewcommand{\cong}{\simeq}

\begin{document}

\title[Stability conditions under change of base field]{Stability conditions under change of base field}
\author[P.\ Sosna]{Pawel Sosna}

\address{Dipartimento di Matematica ``F.\ Enriques'', Universit\`a degli Studi di Milano, Via Cesare Saldini 50,
20133 Milano, Italy}
\email{pawel.sosna@guest.unimi.it}

\begin{abstract} \noindent
We investigate the behaviour of Bridgeland stability conditions under change of base field with particular focus on the case of finite Galois extensions. In particular, we prove that for a variety $X$ over a field $K$ and a finite Galois extension $L/K$ the stability manifold of $X$ embeds as a closed submanifold into the stability manifold of the base change variety.

\vspace{-2mm}\end{abstract}
\maketitle

\maketitle
\let\thefootnote\relax\footnotetext{This work was supported by the SFB/TR 45 `Periods,
Moduli Spaces and Arithmetic of Algebraic Varieties' of the DFG
(German Research Foundation)}

\section{Introduction}
Stability conditions on triangulated categories were introduced by Bridgeland in \cite{Bri1}. He proved that under mild assumptions the set of stability conditions forms a (possibly infinite-dimensional) complex manifold. If one considers so called numerical stability conditions on the bounded derived category of a smooth projective variety, then the stability manifold is always finite-dimensional. Therefore in this paper, apart from section five, only numerical stability conditions will be considered.\smallskip

The stability manifold always lives over the complex numbers, even if one considers the bounded derived category of a smooth projective variety defined over a, say, finite field. Thus, considering a field extension $L/K$ and a smooth projective variety $X$ over $K$ it is interesting to ask how the stability manifolds of $X$ and its base change variety $X_L$ are related. The description of the topology on the manifolds suggests that ${\rm{Stab}}(X)$ and ${\rm{Stab}}(X_L)$ might be different if their numerical Grothendieck groups are. On the other hand one might expect that if the numerical Grothendieck group does not change under scalar extension, then neither does the stability manifold. In order to tackle these questions we will for the most part assume that the field extension is finite and Galois.\smallskip

In the literature there are several examples of smooth projective varieties where the numerical stability manifold is at least partially known. The investigated cases include $\IP^1$ (see \cite{Okada}), curves (see \cite[Thm.\ 9.1]{Bri1} and \cite{Macri1}) and varieties with complete exceptional collections, e.g.\ projective spaces (see \cite{Macri1}). Although all results are formulated over the complex numbers, inspection of the arguments shows that the numerical stability manifold in fact does not depend on the ground field. An explanation for this phenomenon seems to be that in the above cases the structure of the derived category and/or the numerical Grothendieck group is particularly simple. Another prominent example is the case of K3 and abelian surfaces dealt with in \cite{Bri2}. Here the situation is more complicated and even the formulation of the results over other fields than $\IC$ requires certain adjustments, e.g.\ the replacement of the integral cohomology 
 groups by Chow groups (modulo numerical equivalence). Since in this example the numerical Grothendieck group can change under scalar extension, one can indeed expect the stability manifold to depend on the ground field.\smallskip

In this note we try to provide some general insight into the behaviour of stability conditions under scalar extension. To do this we first have to recall the basics of the theory of stability conditions. We then investigate the naturally defined maps between the stability manifolds ${\rm{Stab}}(X)$ and ${\rm{Stab}}(X_L)$ and prove 
\smallskip

\noindent {\bf Main result 1 (Corollary \ref{closed-submf})} {\it For any finite and separable field extension $L/K$ the manifold ${\rm{Stab}}(X)$ is a closed submanifold of ${\rm{Stab}}(X_L)$.}
\smallskip

It is a fundamental fact in the theory that a stability condition can be viewed in two different ways: Via slicings or via hearts of bounded t-structures. In all of the above the first point of view is used. In section four we consider the situation from the second point of view. Somewhat surprisingly it turns out that computations via this approach are fairly elusive.\smallskip

In the last section we study the behaviour of the stability manifold under finite Galois extension if one assumes that the numerical Grothendieck group $N(X)$ does not change. In particular, we apply the following statement to the case of K3 surfaces:
\smallskip

\noindent {\bf Main result 2 (Corollary \ref{stab-iso})} {\it Let $L/K$ be a finite Galois extension. If the map 
\[N(X)\otimes \IC \rightarrow N(X_L)\otimes \IC\]
induced by the pullback map is an isomorphism, ${\rm{Stab}}(X)$ is non-empty and ${\rm{Stab}}(X_L)$ is connected, then we have a homeomorphism ${\rm{Stab}}(X) \cong {\rm{Stab}}(X_L)$.}
\smallskip

\noindent{\bf Acknowledgements.} This paper is a part of my Ph.D. thesis supervised by Daniel \linebreak Huybrechts whom I would like to thank for a lot of fruitful discussions. I am also grateful to Emanuele Macr\`i and Paolo Stellari for their comments on a preliminary version of this paper. Furthermore, I thank the referee for valuable suggestions.

\section{Stability conditions}

We recall basic definitions and properties of Bridgeland's framework. Throughout $\kt$ will be an essentially small triangulated category (see \cite{GM}) which is linear over a field $K$, that is, the morphisms of $\kt$ have the structure of a vector space over $K$ such that the composition law is bilinear. We will furthermore assume that $\kt$ is of finite type, i.e.\ the $K$-vector space $\oplus_i \Hom_\kt(E,F[i])$ is finite-dimensional for any pair of objects $E, F$ in $\kt$.
\begin{definition}\label{def-stabcond}
A \emph{stability condition} $\sigma=(Z,\kp)$ on $\kt$ consists of a group homomorphism $Z\colon  K(\kt) \rightarrow \IC$, where $K(\kt)$ is the Grothendieck group of $\kt$, and a collection of full additive subcategories $\kp(\phi)\subset \kt$ for $\phi \in \IR$, satisfying the following conditions:\newline
\noindent
(SC1) If $0\neq E \in \kp(\phi)$, then $Z(E)=m(E)\exp(i\pi \phi)$ for some $m(E)>0$.\newline
\noindent
(SC2) $\kp(\phi)[1]=\kp(\phi+1)$ for all $\phi$.\newline
\noindent
(SC3) For $\phi_1 > \phi_2$ and $E_i \in \kp(\phi_i)$ we have $\Hom(E_1,E_2)=0$.\newline
\noindent
(SC4) For any $0\neq E \in \kt$ there exist finitely many real numbers $\phi_1 > \ldots > \phi_n$ and a collection of triangles $E_{i-1} \rightarrow E_i \rightarrow A_i$, $i \in \left\{ 1, \ldots n \right\}$, with $E_0=0$, $E_n=E$ and $A_i \in \kp(\phi_i)$.
\end{definition}
A stability condition is called \emph{numerical}, if $Z$ factors over the numerical Grothendieck group $N(\kt)$, which by definition is the quotient of $K(\kt)$ by the nullspace of the Euler form $\chi(E, F)=\sum_i (-1)^i\dim_K\Hom(E, F[i])$. 

The map $Z$ is called the \emph{central charge}, the collection $\kp$ \emph{the slicing} and an object in $\kp(\phi)$ \emph{semistable} of phase $\phi$. Given any interval $I \subset \IR$ one defines $\kp(I)$ to be the extension-closed subcategory of $\kt$ generated by $\kp(\phi)$ for $\phi \in I$. One can then prove that the category $\kp(> t)$ defines a bounded t-structure on $\kt$ for all $t \in \IR$ and that for any $t \in \IR$ the category $\kp(t, t+1]$ is the heart of this t-structure and hence abelian. As a matter of convention, one defines the heart of the slicing $\kp$ to be the abelian category $\kp(0,1]$. Note that one could also consider the bounded t-structure $\kp(\geq t)$, which gives the heart $\kp[t,t+1)$. In fact, we will use a heart of the form $\kp[0,1)$ in section three.

The collection of triangles in (iv) is the \emph{Harder--Narasimhan filtration} of $E$ and the $A_i$ are the semistable factors. The HN-filtration is unique up to a unique isomorphism. Given an object $E \in \kt$ one defines $\phi_\sigma^+(E)=\phi_1$, $\phi_\sigma^-(E)=\phi_n$ and the \emph{mass} of $E$ to be the number $m_\sigma(E)=\sum_{i=1}^n |Z(A_i)|$.

There is an equivalent way of giving a stability condition. To do this define a \emph{stability function} on an abelian category $\ka$ to be a group homomorphism $Z\colon  K(\ka) \rightarrow \IC$ such that for all $0\neq E \in \ka$ the complex number $Z(E)$ lies in the space 
$H:=\left\{ r\exp(i\pi \phi)\; | \; r>0 \; {\rm{and}} \; 0 < \phi \leq 1\right\} \subset \IC$. The \emph{phase} of an object $E \in \ka$ is then defined to be
\[\phi(E)=\frac{1}{\pi} \arg(Z(E)) \in (0,1].\]
Again, note that one could change the definition of a stability function so that the phase lies, for example, in the interval $[0,1)$.

The phase allows one to order objects of $\ka$ and it is thus possible to define semistable objects and HN-filtrations, generalising the classical case, where $\ka$ is the category of coherent sheaves on a smooth projective curve and the ordering is done with respect to the slope $\mu=\deg/\rk$.

A stability function $Z$ is said to have the \emph{Harder--Narasimhan property} if any object possesses a HN-filtration. In this case we will call $Z$ a stability condition. The following proposition provides a nice criterion for checking the HN-property and will be used later on.

\begin{prop}\label{HN-abel}\emph{\cite[Prop.\ 2.4]{Bri1}}
Suppose a stability function $Z\colon K(\ka) \rightarrow \IC$ satisfies the chain conditions \newline
(a) there are no infinite sequences of subobjects in $\ka$
\[
 \cdots \subset E_{i+1} \subset E_i \subset \cdots \subset E_2 \subset E_1 
\]
with $\phi(E_{i+1}) > \phi(E_i)$ for all $i$, \newline
(b) there are no infinite quotients in $\ka$
\[
 E_1 \twoheadrightarrow E_2 \twoheadrightarrow \cdots E_i \twoheadrightarrow E_{i+1} \twoheadrightarrow \cdots
\]
with $\phi(E_i) > \phi(E_{i+1})$ for all $i$.
Then $\ka$ has the Harder-Narasimhan property.
\end{prop}

The connection between the two concepts is given by
\begin{prop}\label{eq-st-cond} \emph{\cite[Prop.\ 5.3]{Bri1}}
To give a stability condition on a triangulated category is equivalent to giving the heart $\ka$ of a bounded t-structure and a stability function with HN-property on $\ka$.
\end{prop}

The proof is roughly as follows. Given $\sigma=(Z, \kp)$ one sets $\ka=\kp(0,1]$ and verifies that $Z$ defines a stability function on $\ka$. One then checks that the corresponding semistable objects are the nonzero objects of the categories $\kp(\phi)$ for $0<\phi \leq 1$. The decompositions of condition (SC4) in Definition \ref{def-stabcond} are Harder--Narasimhan filtrations.

In the other direction, given $(Z, \ka)$ one defines $\kp(\phi)$ for $\phi \in (0,1]$ to be the semistable objects of phase $\phi$ with respect to $Z$. Condition (SC2) then defines $\kp(\phi)$ for all $\phi \in \IR$ and (SC3) is easily checked. Filtrations as in (SC4) are obtained by combining the filtrations with respect to the t-structure with the Harder--Narasimhan filtrations given by $Z$. 

Note that if one works with a stability function on $\ka$ whose phase is e.g.\ in the interval $[0,1)$, it is of course necessary to slightly modify the argument.

In order to exclude fairly pathological examples one only considers \emph{locally finite} stability conditions. By definition, $\sigma=(Z, \kp)$ is locally finite if there exists some $\epsilon > 0$ such that for all $\phi \in \IR$ the category $\kp(\phi-\epsilon, \phi+\epsilon)$ is of finite length, i.e.\ Artinian and Noetherian. The same definition applies to the above presented point of view via hearts.
\smallskip

\noindent
{\bf{Convention:}} From here on all stability conditions will be locally finite. The set of locally finite stability conditions will be denoted by ${\rm{Stab}}(\kt)$.

Bridgeland introduces a generalised metric on ${\rm{Stab}}(\kt)$ (\cite[Prop.\ 8.1]{Bri1}):
\begin{equation}
\displaystyle d(\sigma_1,\sigma_2)=\sup_{0\neq E} \left\{ |\phi^-_{\sigma_1}(E)-\phi^-_{\sigma_2}(E)|, |\phi^+_{\sigma_1}(E)-\phi^+_{\sigma_2}(E)|, |\log\frac{m_{\sigma_1}(E)}{m_{\sigma_2}(E)}|\right\} \in [0, \infty] 
\end{equation}
and proves the

\begin{thm}\emph{\cite[Prop.\ 1.2]{Bri1}} \label{stab-mfd}
For each connected component $\Sigma \subset {\rm{Stab}}(\kt)$ there is a linear subspace $V(\Sigma) \subset \Hom_\IZ(K(\kt), \IC)$ with a linear topology  and a local homeomorphism $\mathcal{Z}\colon \Sigma \rightarrow V(\Sigma)$ sending a stability condition $(Z, \kp)$ to its central charge $Z$.
A similar result holds if one considers ${\rm{Stab}}_N(\kt)$, the set of numerical stability conditions, i.e. one substitutes $K(\kt)$ by $N(\kt)$. 
\end{thm}

In particular, if $K(\kt)\otimes \IC$ is finite-dimensional, then ${\rm{Stab}}(\kt)$ is a finite-dimensional manifold.

There are two groups acting on the stability manifold: The group of exact autoequivalences $\Aut(\kt)$ and $\widetilde{\text{GL}}_+(2,\IR)$, the universal cover of $\text{GL}_+(2,\IR)$. The former acts from the left by isometries as follows. For $\sigma=(Z, \kp)$ and $\Phi \in \Aut(\kt)$ we set $\Phi(\sigma)=(Z\circ \Phi^{-1}, \kp')$ with $\kp'(t)=\Phi(\kp(t))$. 

For the second action first recall that $\widetilde{\text{GL}}_+(2,\IR)$ can be thought of as pairs $(T, f)$, where $f\colon  \IR \rightarrow \IR$ is an increasing map with $f(\phi+1)=f(\phi)+1$, and $T\colon  \IR^2 \rightarrow \IR^2$ is an orientation-preserving linear automorphism, such that the induced maps on $S^1=\IR / 2\IZ=\IR^2 / \IR_{> 0}$ are the same. (Roughly, this can be seen as follows. It is easy to see that the projection map from $C:=\left\{(T,f) \; | f\colon \IR\rightarrow \IR \;\text{increasing map}\ldots \right\}$ to $\text{GL}_+(2,\IR)$ has fibres isomorphic to $\IZ$. Furthermore, we know that $\text{GL}_+(2,\IR)$ is homotopy equivalent to $\text{SO}(2,\IR)$ (for example, because the latter is its maximal compact subgroup). It follows that $C$ is simply connected and, therefore, indeed isomorphic to the universal covering.)
Now, for a $\sigma \in \Stab(\kt)$ and $(T, f) \in \widetilde{\text{GL}}_+(2,\IR)$ define a new stability condition $\sigma'=(Z', \kp')$ by setting $Z'=T^{-1}\circ Z$ and $\kp'(\phi)=\kp(f(\phi))$.\smallskip

\noindent
{\bf{Notation:}} If $Y$ is a smooth projective variety and $\Db(Y)$ its bounded derived category of coherent sheaves, we will write ${\rm{Stab}}(Y)$ for the manifold of numerical locally finite stability conditions ${\rm{Stab}}_N(\Db(Y))$. The Grothendieck group $K({\rm{Coh}}(Y))= K(\Db(Y))$ will be denoted by $K(Y)$ and the numerical Grothendieck group by $N(Y)$.\smallskip

\noindent
{\bf{Convention:}} From here on, unless explicitly stated otherwise, we will only consider numerical stability conditions. 
     
\section{Base change via slicings}
Consider an arbitrary field extension $L/K$, a smooth projective variety $X$ over $K$, the base change scheme $X_L$ over $L$ and the flat projection $p\colon X_L \rightarrow X$ which yields the exact faithful functor $p^*\colon \Db(X) \rightarrow \Db(X_L)$. Given a stability condition $\sigma=(Z, \kp) \in \text{Stab}(X_L)$ one is tempted to define $p_*(\sigma):=\sigma'=(Z', \kp')$ as
\[
Z'=Z \circ p^*
\]
\[
\kp'(\phi)=\left\{ E \in \Db(X) \, | \, p^*(E) \in \kp(\phi) \right\} \; \; \forall \, \phi \in \IR.
\]
It is very easy to see that $p_*(\sigma)$ satisfies properties (SC1)-(SC3) of the definition of a stability condition. Unfortunately, the Harder--Narasimhan property need not hold: Looking at the definition of $p_*(\sigma)$ we see that an object $E \in \Db(X)$ has a HN-filtration with respect to $p_*(\sigma)$ if and only if the HN-filtration of $p^*(E)$ with respect to $\sigma$ is defined over the smaller field. For a possible counter-example cf.\ Remark \ref{not-desc-heart}.

Thus, this definition does not give a stability condition for arbitrary $\sigma \in \text{Stab}(X_L)$ and, therefore, in general this na\"{\i}ve approach does not give a map $\text{Stab}(X_L) \rightarrow \text{Stab}(X)$. 

\begin{definition}
For a field extension $L/K$ define ${\rm{Stab}}(X_L)_p$ to be the subset of stability conditions on $\Db(X_L)$ having the property that $p_*(\sigma)$ admits HN-filtrations. Thus we have a map
\[ p_*\colon  {\rm{Stab}}(X_L)_p \rightarrow \text{Stab}(X).\]
\end{definition}

\begin{lem}
The map $p_*$ is continuous and for any $\sigma, \tau \in {\rm{Stab}}(X_L)_p$ we have $d(p_*(\sigma), p_*(\tau))\leq d(\sigma, \tau)$. Its domain of definition ${\rm{Stab}}(X_L)_p$ is a closed subset of ${\rm{Stab}}(X_L)$. 
\end{lem}

\begin{proof}
The first assertion follows from \cite[Lem.\ 2.9]{MMS} (the local finiteness is automatic, see \cite[Rem.\ 2.7 (ii)]{MMS}). The second follows from the definition of the generalised metric by noting that the supremum on the left is taken over a smaller class of objects. The last assertion again follows from \cite[Lem.\ 2.8]{MMS}.
\end{proof} 

\begin{remark}\label{not-desc-heart}
Consider a heart $\ka$ of a bounded t-structure $\kd^{\leq 0}$ on $\Db(X_L)$ which is of finite length and such that $\kd^{\leq 0}$ does not descend to a t-structure on $\Db(X)$, i.e.
\[\kc^{\leq 0}= \left\{ E \in \Db(X) \; | \; p^*(E)\in \kd^{\leq 0}\right\}\] 
is not a t-structure on $\Db(X)$. We can define a stability condition on $\ka$ by e.g.\ sending all simple objects to $i$. Thus, $\kp(1/2)=\ka$ and $\kp(\phi)=0$ for all $1/2 \neq \phi \in (0,1]$. The HN-filtration of an object $p^*(E)$ in this example is nothing than the filtration of $p^*(E)$ with respect to the cohomology functors defined by $\ka$. Since by assumption $\ka$ does not descend, there exists an object $E_0 \in \Db(X)$ such that the HN-filtration of $p^*(E_0)$ is not defined over the smaller field. Hence in general the subset ${\rm{Stab}}(X_L)_p$ will not be equal to ${\rm{Stab}}(X_L)$.
\end{remark}

\begin{remark}\label{const}
It is easy to see that $p_*$ is $\widetilde{\text{GL}}_+(2,\IR)$-equivariant. An easy consequence is the following. For any $\sigma=(Z, \kp) \in {\rm{Stab}}(X_L)_p$ the stability condition $\widetilde{\sigma}:=(rZ, \kp)$ is in ${\rm{Stab}}(X_L)_p$ for any $r \in \IR_{>0}$.
\end{remark}

Obviously one would like to to be able to say something about the structure of ${\rm{Stab}}(X_L)_p$, e.g.\ whether this set is always non-empty or connected. To tackle these questions we will further assume that the field extension is finite. We then also have the exact functor $p_*\colon \Db(X_L) \rightarrow \Db(X)$ at our disposal. 

\begin{prop}
For a finite field extension $L/K$ the map $p\colon X_L \rightarrow X$ defines a continuous map 
\[
p^*\colon  {\rm{Stab}}(X) \rightarrow {\rm{Stab}}(X_L).
\]
Here, for a $\sigma'=(Z', \kp') \in {\rm{Stab}}(X)$ we define $p^*(\sigma'):=\sigma=(Z,\kp)$ by $Z=Z'\circ p_*$ and $\kp(\phi)=\left\{ F \in \Db(X_L) \; | \; p_*(F)\in \kp'(\phi)\right\}$ for any $\phi \in \IR$.
\end{prop}

\begin{proof}
Note that we have $p_*(\ko_{X_L})=\ko_X^d$, where $d=[L:K]$. An immediate consequence of this is that for any stability condition $\sigma'=(Z', \kp')$ on $\Db(X)$ one has 
\[
p_*(\ko_{X_L}) \otimes \kp'(\phi)= \ko_X^d \otimes \kp'(\phi) \subset \kp'(\phi) \subset \kp'[\phi, +\infty) 
\]
since the categories $\kp'(\phi)$ are additive (in fact abelian).
Since $p$ is flat, it is trivially of finite Tor dimension and therefore \cite[Cor.\ 2.2.2]{Pol} applies giving that $p^*$ exists and is well-defined (note that in \cite{Pol} the t-structure $\kp'(>t)$ is used, whereas we use $\kp'(\geq t)$, which is also what is needed in the example following Cor.\ 2.2.2 in \cite{Pol}). The continuity follows from \cite[Lem.\ 2.9]{MMS}. 
\end{proof}

\begin{remark}
Similarly to $p_*$ the map $p^*$ satisfies $d(p^*(\sigma'), p^*(\tau'))\leq d(\sigma', \tau')$
for all $\sigma', \tau' \in {\rm{Stab}}(X)$.
\end{remark}

\begin{remark}
We can describe geometric analogues of the maps $p_*$ and $p^*$: Let $\pi\colon Y \rightarrow Z$ be a finite unramified covering of smooth projective varieties and consider $\mu$-semistability of coherent sheaves on $Z$ resp.\ on $Y$ with respect to $\ko_Z(1)$ resp.\ $\ko_Y(1):=\pi^*(\ko_Z(1))$. Then it is well-known that a coherent sheaf $F$ on $Z$ is $\mu$-semistable if and only if $\pi^*(F)$ is $\mu$-semistable, cf.\ \cite[Lem.\ 3.2.2]{HL}.
On the other hand a coherent sheaf $F'$ on $Y$ is $\mu$-semistable if and only $\pi_*(F')$ is $\mu$-semistable, cf.\ \cite[Prop.\ 1.5]{Takemoto}.

Note that the classical notion of semistability behaves well under field extensions, see \cite[Thm.\ 1.3.7 and Cor.\ 1.3.8]{HL}. For example, the HN-filtration of $p^*(E)$, where $E$ is a (pure) sheaf on $X$, is the pullback of the HN-filtration of $E$, a fact that will be used later on.
\end{remark}

Recall that the group of automorphisms of $X_L$ acts on ${\rm{Stab}}(X_L)$. In particular, we have an action of $G:=\Aut(L/K)$ on the stability manifold (note, however, that elements of $G$ are only $K$-linear).

We can describe the image of $p^*$ by formulating the

\begin{prop}\label{image-beta}
Let $L/K$ be a finite extension and let $\sigma'$ be an element of $\Stab(X)$. Then $p^*(\sigma')$ is invariant under the action of the group $G=\Aut(L/K)$.
\end{prop}

\begin{proof}
First, note that $(g^{-1})^*=g_*$.  If $g \in G$ and $\sigma=p^*(\sigma')$, then $g(\sigma)=(Z \circ g_*, g^*(\kp))$. It follows that for $E \in \kp(\phi)$ one has $p_*(E)=p_*(g^*(E))$. Furthermore, for any $E \in \Db(X_L)$ the following holds: 
\[Z(g_*(E))=Z' p_*(g_*(E))=Z'(p_*(E))=Z(E).\]
We conclude that $g(\sigma)=\sigma$ as claimed.
\end{proof}

\begin{remark}
Clearly the statement of the proposition is only interesting in the case when $G$ is non-trivial, e.g.\ for a finite Galois extension.
\end{remark}

\begin{remark}\label{comp-galois}
In the case of a finite Galois extension of degree $d$  one has $p_* p^*(E)=E^{\oplus d}$ for any $E \in \Db(X)$ and $p^* p_*=\sum_{g \in G} g^*$, cf.\ e.g.\ \cite{Roberts}. 
\end{remark}

\begin{lem}\label{alphabeta}
If the field extension $L/K$ is finite and Galois of degree $d$, then the composition $p_* \circ p^*\colon  {\rm{Stab}}(X) \rightarrow {\rm{Stab}}(X)$ is equal to the action of $h:=(\frac{1}{d}\cdot \id, \id) \in \widetilde{\rm{GL}}_+(2,\IR)$.In particular, $p^*({\rm{Stab}}(X))$ is contained in ${\rm{Stab}}(X_L)_p$ (see also Proposition \ref{dom-def} below).
\end{lem}

\begin{proof}
Consider a stability condition $\sigma'=(Z', \kp')$ in ${\rm{Stab}}(X)$.
Then $p_* p^*(\sigma')=\sigma''=(Z'', \kp'')$ is defined as
\[Z''= Z'\circ p_* \circ p^*=d Z'\]
\[ \kp''(\phi)=\left\{ F \in \Db(X_L) \, | \, p_*p^*(F)=F^{\oplus d} \in \kp'(\phi) \right\}\]
Clearly $\kp'(\phi) \subset \kp''(\phi)$. Let us prove the other inclusion. Assume $F \in \kp''(\phi)$. Since $\sigma'$ is a stability condition, $F$ has a HN-filtration given by certain triangles $F_{i-1} \rightarrow F_i \rightarrow A_i$, $i \in \left\lbrace  1, \ldots, n\right\rbrace $. Since the direct sum of triangles is a triangle, we can take the $d$-fold direct sum of these and get a filtration of $F^{\oplus d}$. But HN-filtrations are unique and by assumption $F^{\oplus d} \in \kp'(\phi)$, so its HN-filtration is trivial. 
Therefore $n=1$, $\phi_1=\phi$ and $F \in \kp'(\phi)$. Thus, $p_* \circ p^*(Z', \kp')=(dZ', \kp')$ as claimed.
\end{proof}

\begin{cor}\label{beta-inj}
For a finite Galois extension, the map $p_* \circ p^*$ is a homeomorphism, thus $p^*$ is injective and $p_*$ surjective.\qqed
\end{cor}

We will now investigate the domain of definition for the morphism $p_*$. To do this we need the following

\begin{lem}\label{linearised}
If $G$ is a finite group acting on a variety $Y$ over a field $K$ of characteristic prime to the order of the group, then any linearised object in $\Db(Y)$ is isomorphic as a complex to a complex of $G$-linearised sheaves (cf.\ \cite{BriMa} and \cite{Ploog}).  
\end{lem}

\begin{proof}
First recall that a linearised object is a pair $(E,\lambda)$, where $E \in \Db(Y)$ and $\lambda$ is a collection of isomorphisms $\lambda_g\colon  E \rightarrow g^*(E)$ satisfying the usual cocycle condition, and morphisms between two such pairs are morphisms in $\Db(Y)$ which are compatible with the linearisations (note that the same definition applies, for example, in an abelian category). Denote the category of linearised objects by $\kt$ and write $D^G(Y)=\Db({\rm{Coh^G}}(Y))$ for the bounded derived category of the abelian category ${\rm{Coh^G}}(Y)$ of linearised coherent sheaves on $Y$. Clearly, there is a functor $\Phi\colon  D^G(Y) \rightarrow \kt$, which was proved to be an equivalence in \cite{Ploog}. We will use that it is fully faithful, but to make the paper more self-contained we will show the statement of the lemma (which is weaker than essential surjectivity) by induction on the number of cohomology objects. The case $n=1$ is obvious. Let $(E, \lambda)$ be a complex with $n$ cohomology objects. We may assume $H^i(E)=0$ for $i \geq 2$. Consider the triangle given by the standard t-structure on $\Db(Y)$:
\[ \tau^{\leq 0}(E)=:E' \rightarrow E \rightarrow \tau^{\geq 1}(E)=H^1(E)[-1]\rightarrow \tau^{\leq0}(E)[1]. \]
For any $\lambda_g$ we get, since truncation is a functor and commutes with $g^*$ for any $g \in G$, corresponding morphisms on $E'$ and $H^1(E)$ and these morphisms define linearisations of these complexes. By induction $E' \cong \Phi(F)$ and $H^1(E)[-1]\cong \Phi(F')$ in $\kt$ for some complexes $F, F' \in D^G(Y)$. Since the morphisms in the triangle and the isomorphisms are compatible with the linearisations, the map $F'[-1] \rightarrow F$ is a map in $D^G(Y)$ and hence a cone is in $D^G(Y)$. This cone is isomorphic to $E$ in $\Db(Y)$. This concludes the proof.     
\end{proof}
\smallskip
\noindent
{\bf{Convention:}} From here on we assume that the order of the Galois group does not divide the characteristic of the ground field.

\begin{remark}
Note that our proof does not show $\Phi$ to be essentially surjective, since the isomorphism does not need to respect the linearisations of $E$ resp.\ the cone.

One could try to generalise the above statement to arbitrary faithfully flat morphisms replacing linearisations of sheaves resp.\ complexes by descent data. Note that the fully faithfulness used above is satisfied in this situation.
\end{remark}

We also need the following 

\begin{lem}\label{betaalpha}
Let $L/K$ be finite and Galois and consider $\sigma=(Z, \kp) \in {\rm{Stab}}(X_L)_p$. Then $p^* \circ p_*(\sigma)=(\widetilde{Z}, \widetilde{\kp})$, where
\[
\widetilde{Z}(E)=\sum_{g \in G} Z(g^*(E)) \; \; \; \; \;and
\]  
\[
\widetilde{\kp}(\phi)=\left\{E \; | \; \oplus_{g \in G} g^*(E) \in \kp(\phi) \right\}=\left\{E \; | \;  \left\{g^*(E)\right\}_{g \in G} \subset \kp(\phi) \right\}.
\]
\end{lem}

\begin{proof}
Using Remark \ref{comp-galois} we immediately get the formula for $\widetilde{Z}$ and the first equality for $\widetilde{\kp}(\phi)$. As to the second equality: ``$\subset$'' holds because the categories $\kp(\phi)$ are closed under direct summands and ``$\supset$'' holds because they are additive.
\end{proof}

We can now prove the 

\begin{prop}\label{dom-def}
For a finite Galois extension $L/K$ with Galois group $G$ a stability condition $\sigma=(Z,\kp)$ is in ${\rm{Stab}}(X_L)_p$ if and only if its slicing $\kp$ is invariant under the action of the Galois group, i.e.\ $E \in \kp(\phi)$ implies $g^*(E) \in \kp(\phi)$ for all $g \in G$. In particular, the subset of $G$-invariant stability conditions is contained in ${\rm{Stab}}(X_L)_p$.
\end{prop} 

\begin{proof}
If $p_*(\sigma)$ is a stability condition on $\Db(X)$, then we can consider $\widetilde{\sigma}=p^* \circ p_*(\sigma) \in \text{Stab}(X_L)$. Lemma \ref{betaalpha} shows that the slicing $\widetilde{\kp}$ of $\widetilde{\sigma}$ is invariant under the action of the Galois group and that $\widetilde{\kp}(\phi) \subset \kp(\phi)$ for all $\phi \in \IR$. In fact, we have an equality. To see this, consider $\phi \in \IR$ and $E \in \kp(\phi)$. We know that $E$ has a HN-filtration with respect to $\widetilde{\sigma}$, i.e.\ it can be filtered by objects in $\widetilde{\kp} \subset \kp$. Since $E$ is semistable with respect to $\sigma$, this filtration has to be trivial. Thus, $E \in \widetilde{\kp}(\phi)$.

For the converse implication assume the slicing $\kp$ of $\sigma$ to be $G$-invariant. Consider the HN-filtration of an object $p^*(E)$, $E \in \Db(X)$. Applying an arbitrary element $g \in G$ yields the HN-filtration of $g^*p^*(E)=p^*(E)$ with respect to $\sigma$, because the slicing is $G$-invariant. It follows that all the objects of the filtration are linearised objects of $\Db(X_L)$, because HN-filtrations are unique up to unique isomorphism. By Lemma \ref{linearised} any such object is isomorphic to a complex of $G$-equivariant sheaves on $X_L$. Using Galois descent we see that a complex of equivariant objects is defined over $K$, and hence the HN-filtration is defined over $K$. 
\end{proof}

\begin{remark}
The stability function of a stability condition $\sigma \in {\rm{Stab}}(X_L)_p$ need not be $G$-invariant, since we only assume that the phase is constant on the orbits of semistable objects under the action of $G$. For an unstable object $E$ the numbers $Z(E)$ and $Z(g^*(E))$ ($g \in G$) will in general not even have the same phase.
\end{remark}

Using Lemma \ref{betaalpha} we see that the restriction of $p^* \circ p_*$ to the subset of $G$-invariant stability conditions
\[
{\rm{Stab}}(X_L)^G=\left\{ \sigma \in {{\rm{Stab}}}(X_L) \, | \, g \sigma= \sigma \; \forall g \in G\right\}
\] 
is equal to $h=(\frac{1}{d}\cdot \id, \id)$. In particular, the map $p^*\colon  {\rm{Stab}}(X) \rightarrow {\rm{Stab}}(X_L)^G$ is surjective. Hence, $p^*$ is a homeomorphism, since we already have seen that it is injective. Thus, we have the following diagram
\[
\begin{xy}
\xymatrix{ {\rm{Stab}}(X_L)_p \ar[rd]^{p_*} & \ar@{_{`}->}[l]{\rm{Stab}}(X_L)^G \\
					&  {\rm{Stab}}(X). \ar[u]^{p^*}_\cong}
\end{xy}
\]
Using this and the fact that the subset ${\rm{Stab}}(X_L)^G$ is a closed submanifold of ${\rm{Stab}}(X_L)$ (cf.\ \cite[Lem. 2.15]{MMS} where this result is formulated for the action of a finite group on a smooth projective complex variety but the proof does not use these assumptions) we immediately derive

\begin{cor}\label{closed-submf}
For any finite and separable field extension $L/K$ the map $p^*$ realizes ${\rm{Stab}}(X)$ as a closed submanifold of ${\rm{Stab}}(X_L)$.
\end{cor}

\begin{proof}
The statement is clear for a finite Galois extension from the above discussion. The general case follows by considering the tower $K \subset L \subset L^n$, where $L^n$ denotes the normal closure, the induced commutative diagram
\[
\begin{xy}
\xymatrix{ {\rm{Stab}}(X) \ar[rd]_{p^*} \ar[rr]^{\widehat{p}^*} && {\rm{Stab}}(X_{L^n}) \\
 & {\rm{Stab}}(X_L) \ar[ru]_{\overline{p}^*}}
\end{xy}
\]
(where $p\colon  X_L \rightarrow X$, $\overline{p}\colon  X_{L^n} \rightarrow X_L$ and $\widehat{p}\colon  X_{L^n} \rightarrow X$ are the projections) and the fact that $L \rightarrow L^n$ is Galois.
\end{proof}

\begin{cor}
Let $L/K$ be a finite and separable extension and $L^n$ the normal closure of $L$. Consider the diagram

\[
\begin{xy}
\xymatrix{ \Stab(X_{L^n})_{\widehat{p}} \ar@/^1.5pc/[rr]^{\widehat{p}_*} \ar@{^{(}->}[r] & {\rm{Stab}}(X_{L^n}) & \Stab(X) \\
						& \Stab(X_L) \ar[u]^{\overline{p}^*} & \Stab(X_L)_p \ar@{_{(}->}[l]\ar[u]^{p_*} }
\end{xy}
\]

(where the notation is as in the proof of Corollary \ref{closed-submf}).
Then for $\sigma \in {\rm{Stab}}(X_L)$ the following holds
\[
\sigma \in {\rm{Stab}}(X_L)_{p} \Longleftrightarrow \overline{p}^* (\sigma) \in {\rm{Stab}}(X_{L^n})_{\widehat{p}}.
\]

\end{cor}

\begin{proof}
If $\sigma \in {\rm{Stab}}(X_L)_{p}$, then applying Lemma \ref{alphabeta} to the pair $\overline{p}_*$ and $\overline{p}^*$ and using Remark \ref{const} we see that $\overline{p}_*\overline{p}^*(\sigma) \in {\rm{Stab}}(X_L)_{p}$, therefore $p_*\overline{p}_*\overline{p}^*(\sigma)=\widehat{p}_*\overline{p}^*(\sigma)$ exists and hence $\overline{p}^*(\sigma)\in {\rm{Stab}}(X_{L^n})_{\widehat{p}}$. The converse is proved similarly.
\end{proof}

Until the end of this section we will work in the case of a finite Galois extension of degree $d$. We know from Proposition \ref{dom-def} that the $G$-invariant stability conditions ${\rm{Stab}}(X_L)^G$ are contained in the set ${\rm{Stab}}(X_L)_p$. The next two results establish geometric connections between the two sets.

\begin{lem}
The subset ${\rm{Stab}}(X_L)^G$ is a retract of ${\rm{Stab}}(X_L)_p$. 
\end{lem}

\begin{proof}
Recall that if $i\colon  S \subset T$ is a pair of topological spaces, then by definition $S$ is called a retract of $T$ if there exists a map $f\colon  T \rightarrow S$ such that $f\circ i=\id_S$. 

Define a map $f\colon  {\rm{Stab}}(X_L)_p \rightarrow {{\rm{Stab}}(X_L)^G}$ as follows.
For $\sigma=(Z,\kp) \in {\rm{Stab}}(X_L)_p$ we set $f(\sigma):=\sigma'=(Z', \kp')$, where
$Z'(E)=(1/d) \sum_{g \in G} Z(g^*(E))$ and $\kp'=\kp$. It is fairly easy to show that $\sigma'$ is a stability condition: Since by Proposition \ref{dom-def} the slicing of $\sigma$ is $G$-invariant, the HN-filtrations of $\sigma$ and $\sigma'$ coincide. Furthermore, any object in $\kp(\phi)$ clearly still has phase $\phi$, since $Z(g^*(E))$ is of phase $\phi$, for all $g \in G$. 

Next, one has to verify that $f$ is continuous: Let $\sigma$ and $\tau$ be two stability conditions. The distance between $f(\sigma)$ and $f(\tau)$ is given by formula (2.1) and since the HN-filtrations do not change, neither do the first two numbers in the expression. As for the last one note that for any semistable object $A$ and any $g \in G$ one has $Z(g^*(A))=\delta_g Z(A)$ for some $\delta_g > 0$ and hence $|Z(A)+Z(g^*(A))|=|Z(A)|+|Z(g^*(A))|$. Therefore 
\[
m_{f(\sigma)}(E)=\frac{1}{d}\big(\sum_{g \in G} m_\sigma(g^*(E))\big)
\]
for an arbitrary object $E$ and hence the supremum over all $E$ of $|\log\frac{m_{f(\sigma)}(E)}{m_{f(\tau)}(E)}|$ is small if the same holds for $|\log\frac{m_\sigma(E)}{m_\tau(E)}|$.

Denoting the embedding of ${{\rm{Stab}}(X_L)^G}$ into ${\rm{Stab}}(X_L)_p$ by $i$ we immediately conclude that the composition $f \circ i\colon {{\rm{Stab}}(X_L)^G}\, \hookrightarrow {\rm{Stab}}(X_L)_p \rightarrow {{\rm{Stab}}(X_L)^G}$ is equal to the identity map and we therefore have a retraction. 
\end{proof}

In fact, a stronger statement is true:
\begin{prop}\label{defo-retract}
The subset ${\rm{Stab}}(X_L)^G$ is a deformation retract of ${\rm{Stab}}(X_L)_p$.
\end{prop}

\begin{proof}
Recall that $S \subset T$ is a deformation retract if there exists a homotopy $H\colon  T\times [0,1] \rightarrow T$ such that $H(-,0)=\id_T$, $H(-,1) \subset S$ and $H(s,1)=s$ for any $s \in S$. For details see \cite{Dold}.

The strategy of the proof is the following. Consider an element $\sigma=(Z,\kp)$ in ${\rm{Stab}}(X_L)_p$. Since $\kp$ is already $G$-invariant, it seems natural to only deform $Z$ until it also becomes $G$-invariant.

Assume for simplicity that $d=2$, so $G=\left\{1,g\right\}$, and consider the map
\[H\colon  {\rm{Stab}}(X_L)_p\times\left[0,1\right] \rightarrow {\rm{Stab}}(X_L)_p\]
sending $(\sigma, t)$ to $(\widetilde{Z},\widetilde{\kp})$, where $\widetilde{Z}(E)=Z(E)+tZ(g^*(E))$ and $\widetilde{\kp}=\kp$. Clearly $H(\sigma, 0)=\sigma$ and $H(\sigma, 1) \in {{\rm{Stab}}(X_L)^G}$. Of course, with this definition $H(-,1)\neq \id$ on ${{\rm{Stab}}}(X_L)^G$, but this is easily solved: Since for a $(Z',\kp')=\sigma' \in {{\rm{Stab}}}(X_L)^G$ we have the equality $H(\sigma',1)=(2Z',\kp')$, it is easy to write down a homotopy from $H(-,1)$ to the identity map on ${{\rm{Stab}}}(X_L)^G$. The more challenging issue is the continuity of $H$. Inspecting the proof of the previous lemma it is easy to see that $H(-,t)\colon  {\rm{Stab}}(X_L)_p \rightarrow {\rm{Stab}}(X_L)_p$ is continuous, for any $t \in \left[0,1\right]$. Now fix a $\sigma$ in ${\rm{Stab}}(X_L)_p$ and consider for simplicity only the question of continuity in $0$. Looking at the definition of the metric on ${{\rm{Stab}}}(X_L)$ (equation (2.1)), we see that the first two terms are zero, because the HN-filtrations of $H(\sigma,0)=\sigma$ and $H(\sigma,\epsilon)=:\tau$ are the same. Therefore we only need to consider the last term. Take an arbitrary semistable object $A \in \kp(\phi)$ for some $\phi \in \IR$. We know that $g^*(A)$ is again semistable of the same phase. Now, $m_\sigma(A)=|Z(A)|$ and $m_\tau(A)=(1+\epsilon\cdot \lambda_A)|Z(A)|$, where $\lambda_A=|Z(g^*(A))|/ |Z(A)|$. Thus the quotient $(m_\tau(A))/(m_\sigma(A))$ is $(1+\epsilon\cdot \lambda_A)$ and we need $\lambda_A$ to be bounded, if the last term in the metric is to remain small. Since the numerical Grothendieck group has finite rank and the linear operator $(Z\circ g^*)/Z$ on $N(X_L)\otimes \IC\supset N(X_L)$ has bounded norm, the quotient $|Z(g^*(A))|/ |Z(A)|$ is indeed bounded by some constant $C$. Thus for an arbitrary object $E \in \Db(X_L)$ we get
\[m_\tau(E)/m_\sigma(E)=\frac{\sum_i|Z(A_i)|+\epsilon \sum_i|Z(g^*(A_i))|}{\sum_i|Z(A_i)|}=1+\epsilon\frac{|Z(g^*(A_1))|+\ldots +|Z(g^*(A_n))|}{|Z(A_1)|+\ldots + |Z(A_n)|}\leq 1+\epsilon C \]
and thus $|\log(m_\tau(E)/m_\sigma(E))|$ is small provided $\epsilon$ is small enough. We conclude that $H$ is continuous. For $d>2$ the proof is similar and left to the reader.
\end{proof} 

\begin{remark}
The proofs of all results in this chapter with the exception of Proposition \ref{defo-retract} work for non-numerical locally finite stability conditions.
\end{remark}

\section{Base change via hearts} 

Recall that a stability condition $\sigma=(Z, \kp)$ can also be viewed as a pair consisting of a heart $\kd$ and a stability condition $Z$ on it, cf.\ Proposition \ref{eq-st-cond}. Thus, the stability manifold of a triangulated category $\kt$ is partitioned with respect to the hearts of bounded t-structures in $\kt$. Given a heart $\kd$ in $\Db(X_L)$ we can try to understand whether a stability condition $\sigma=(\kd, Z)$ descends to a stability condition on $\Db(X)$. One could apply a result like this to e.g.\ prove that $\Stab(X)$ is non-empty. In fact, this is precisely what we will do in the case of K3 surfaces at the end of the last section. 

We start with the following  

\begin{prop}\label{induced-sc-abelian}
Let $f\colon \kc \rightarrow \kd$ be an exact faithful functor between abelian categories. If $Z$ is a stability condition on $\kd$, then composition with $F$ induces a ( not necessarily numerical) stability condition $Z'$ on $\kc$.
\end{prop}

\begin{proof}
The composition $Z\circ [\,] \circ F$, where $ [\,] \colon  \kd \rightarrow K(\kd)$ sends an object to its class, is clearly an additive function from $\kc$ to $\IC$ and hence by the universal property of the Grothendieck group we get a group homomorphism $Z'\colon  K(\kc) \rightarrow \IC$. By Proposition \ref{HN-abel} we have to check whether there exists an infinite chain of subobjects/subquotients with increasing phases. By definition of $Z'$ we have $\phi(C)=\phi(F(C))$ for any $C \in \kc$. Assume, for example, that there exists an infinite chain of subobjects with increasing phases in $\kc$. Note that the faithfulness of $F$ is equivalent to the condition $\ker(F)=0$, where  $\ker(F)=\left\lbrace c \in \kc \; : \; F(c)\cong 0\right\rbrace$. Thus, using the triviality of the kernel and the exactness of $F$ one gets an infinite chain in $\kd$ with the same property. Since this is not possible, $Z'$ has the Harder--Narasimhan property and thus is a stability condition. 
\end{proof}

Note that if one wants to prove a similar statement for numerical stability conditions it is necessary to assume that the functor $F$ respects the nullspaces of the Euler forms, which is, for example, true if $F$ admits a right adjoint. The last condition is of course satisfied for the pullback functor.

Consider a bounded t-structure with heart $\kd$ on $\Db(X_L)$. It defines a t-structure with heart $\kc$ on $\Db(X)$ if for any object $E \in \Db(X)$ the filtration with respect to cohomology objects given by $\kd$
\[\begin{xy}
\xymatrix{ 0=F_0 \ar[rr] && F_1 \ar[r] \ar[ld] & \cdots \ar[r] & F_{n-1} \ar[rr] && F_n=p^*(E), \ar[ld]\\
					     & B_1 \ar@{-->}[lu]&& && B_{n-1} \ar@{-->}[lu]}
\end{xy}\]
where $B_i \in \kd[i]$ for all $i$, is defined over the smaller field. Equivalently, the t-structure $\kd^{\leq 0}$ in $\Db(X_L)$ descends if the full subcategory
\[\kc^{\leq 0}=\left\{ E \in \Db(X) \; | \; p^*(E) \in \kd^{\leq 0}\right\}\]
defines a t-structure on $\Db(X)$. 

Assume that a t-structure on $\Db(X_L)$ given by a heart $\kd$ descends to a t-structure on $\Db(X)$, so in particular the category 
\[\kc=\left\{E \in \Db(X) \; | \; p^*(E) \in \kd \right\}\]
is abelian. Denote the set of stability conditions on the abelian category $\kc$ resp.\ $\kd$ by ${\rm{Stab}}(\kc)$ resp.\ ${\rm{Stab}}(\kd)$.

\begin{cor}\label{heart-descent}
There exists a natural map $\alpha\colon  {\rm{Stab}}(\kd) \rightarrow {\rm{Stab}}(\kc)$.
\end{cor}

\begin{proof}
Since short exact sequences in $\kc$ are nothing but distinguished triangles in $\Db(X)$, the pullback functor $F:=p^*\colon  \kc \rightarrow \kd$ is exact. Furthermore, $F$ is faithful, since the stalk of an object $p^*(E)$ in a point $y \in X_L$ equals the stalk of $E$ in $p(y)$ tensorized with $L$. Proposition \ref{induced-sc-abelian} then gives the result.
\end{proof}

\begin{remark}
Note that even if a t-structure $\kd^{\leq 0}$ descends to a t-structure on $\Db(X)$, it is not necessarily true that $\kd^{\leq 0}$ is $\Aut(L/K)$-invariant. A priori one only has that if
\[ \begin{xy} \xymatrix{D_1 \ar[r] & p^*(E) \ar[r] & D_2 \ar[r] & D_1[1]} \end{xy}\]
is the decomposition of $p^*(E)$, where $E \in \Db(X)$, with respect to the t-structure (i.e.\ $D_1 \in \kd^{\leq 0}$ and $D_2 \in \kd^{>0}$), then the objects $D_1$ and $D_2$ are $\Aut(L/K)$-invariant. We cannot conclude this for objects not appearing in the decomposition of some $p^*(E)$.   
\end{remark}

The morphism $\alpha$ a priori does not correspond to the morphism $p_*$ of the previous section. Of course, the stability function $Z'$ is defined in the same way and the pullback of a semistable object $E$ of phase $\phi$ in $\kc$ by definition has the same phase in $\kd$, but $p^*(E)$ is not necessarily semistable. Thus, the following definition is reasonable.

\begin{definition}
For a field extension $L/K$ define ${\rm{Stab}}(X_L)_\alpha$ to be the set of stability conditions $\sigma=(\kd, Z) \in {\rm{Stab}}(X_L)$ such that $\kd$ descends to a heart in $\Db(X)$ (i.e.\ the corresponding t-structure $\kd^{\leq 0}$ descends).
\end{definition}

\begin{remark}
In contrast to ${\rm{Stab}}(X_L)_p$ it seems difficult to show that ${\rm{Stab}}(X_L)_\alpha$ is closed in ${\rm{Stab}}(X_L)$ or that $\alpha\colon  {\rm{Stab}}(X_L)_\alpha \rightarrow {\rm{Stab}}(X)$ is continuous, the problem being that the pullback of the HN-filtration of an object $E \in \Db(X)$ is not necessarily the HN-filtration of $p^*(E)$. What one can say is that ${\rm{Stab}}(X_L)_\alpha$ is preserved by the action of $\Aut(L/K)$. One further property of ${\rm{Stab}}(X_L)_\alpha$ is described in Example \ref{tilting}. 
\end{remark}

Under some additional assumptions we can establish a connection between $\alpha$ and $p_*$.  If $\sigma=(Z,\kp) \in \Stab(X_L)$, we write $\kd$ for the heart $\kp(0,1]$ .

\begin{prop}
Let $L/K$ be a finite Galois extension and let $\sigma=(\kd, Z)$ be a stability condition on $\Db(X_L)$ such that for any $\phi \in (0,1]$ the subcategory of semistable objects in $\kd$ of phase $\phi$ is invariant under the action of the Galois group. Then the heart $\kd$ descends to $\Db(X)$ and $\alpha(Z)$ corresponds to $p_*(\sigma)=\sigma'$ . We therefore have an inclusion ${\rm{Stab}}(X_L)_p \subset {\rm{Stab}}(X_L)_\alpha$ and $\alpha_{|{\rm{Stab}}(X_L)_p}=p_*$.  
\end{prop} 

\begin{proof}
Under our assumptions Proposition \ref{dom-def} shows that $\sigma$ is an element in $\Stab(X_L)_p$ and hence the heart $\kd$ descends as claimed . Thus we only have to show that for a semistable object $E \in \kc$ of phase $\phi \in (0,1]$ the object $p^*(E)$ is again semistable (clearly, if $p^*(E)$ is semistable, then so is $E$). Assuming the converse there exists a semistable object $F \subset p^*(E)$ such that $\phi_\sigma(F) > \phi_\sigma(p^*(E))$. Now, the functor $p_*$ is exact and sends any object $X \in \kd$ to an object in $\kc$, since $\kc= \left\{ Y \in \Db(X) \; | \; p^*(Y) \in \kd \right\}$ and $p^*(p_*(X))=\oplus_{g \in G} g^*(X) \in \kd$. Thus, we can apply $p_*$ to the inclusion $F \subset p^*(E)$ and get an inclusion $p_*(F) \subset p_*p^*(E)=E^{\oplus d}$. Since $E$ is semistable in $\kc$, so is $E^{\oplus d}$ and the phases are equal. But $p_*(F)$ is a destabilizing object of $E^{\oplus d}$ since 
\[\phi_{\sigma'}(p_*(F))=\phi_\sigma(p^*p_*(F))=\phi_\sigma (\oplus_{g \in G}g^*(F))=\phi_\sigma(F)>\phi_\sigma(p^*(E))=\phi_{\sigma'}(E)=\phi_{\sigma'}(E^{\oplus d}).\]
This is a contradiction, therefore an object $E$ is semistable if and only if $p^*(E)$ is semistable. 
\end{proof}

We can say a little bit more about descent of hearts using the theory of tilting. To do this first recall that
a \emph{torsion pair} in an abelian category $\ka$ consists of two full additive subcategories $(\kt, \kf)$ such that for any $T \in \kt$ and $F \in \kf$ we have $\Hom(T, F)=0$ and furthermore for any object $A \in \ka$ there exists an exact sequence
\[
0 \rightarrow T \rightarrow A \rightarrow F \rightarrow 0
\]
with $T \in \kt$ and $F \in \kf$. Note that the exact sequence is unique up to isomorphism.

\begin{ex}\label{tilting}
Assume that $L/K$ is finite Galois, let $\sigma=(\kd, Z)$ be a stability condition on $\Db(X_L)$ such that its heart $\kd$ descends to $\Db(X)$ and assume that there is a torsion pair $(\kt,\kf)$ in $\kd$. Denote the descended heart by $\kc$. We can consider the tilt of $\kd$ with respect to the torsion pair, which gives us a new heart $\kd^\sharp$.
If the torsion pair $(\kt, \kf)$ is invariant under the Galois group (i.e.\ $g^*(T) \in \kt$ and $g^*(F) \in \kf$ for any $T \in \kt$, $F\in \kf$ and $g \in G$), then it is fairly easy to see that the categories
\[ \kt':=\left\{ E \in \kc \; | \; p^*(E) \in \kt \right\} \;\;\; {\rm{and}} \;\;\; \kf':=\left\{ E \in \kc \; | \; p^*(E) \in \kf \right\}\]
define a torsion pair in $\kc$: Clearly $\kt'$ and $\kf'$ satisfy the first requirement of a torsion pair. To see the second consider an arbitrary object $C \in \kc$ and the decomposition of its pullback with respect to the pair $(\kt, \kf)$
\[ \begin{xy}\xymatrix{0 \ar[r] & T \ar[r] & p^*(C) \ar[r] & F \ar[r] & 0.}\end{xy}\]
Applying an arbitrary $g \in G$ to this sequence and using that the torsion pair is invariant, we conclude that the objects $T$ and $F$ are linearised and therefore $T\cong p^*(T')$ and $F\cong p^*(F')$ for some uniquely determined $T' \in \kt'$ and $F' \in \kf'$. Thus, $\kt'$ and $\kf'$ indeed define a torsion pair and this gives a new heart $\kc'$. We therefore get a morphism ${\rm{Stab}}(\kd') \rightarrow {\rm{Stab}}(\kc')$ and conclude that the subset ${\rm{Stab}}(X_L)_\alpha$ is closed under tilting with respect to Galois-invariant torsion pairs.

Explicit examples can be constructed as follows. Let $\kd=\Coh(X_L)$ be the heart of the standard t-structure on $\Db(X_L)$ which clearly descends. We know that Harder--Narasimhan filtrations are stable under base change (cf.\ \cite[Thm.\ 1.3.7]{HL}). Sometimes it is possible to define a torsion pair using the data from the HN-filtration and by the stability of the filtration the torsion pair is Galois-invariant. See the end of the next section for an appearance of this technique in the case of a complex K3 surface and the extension $\IC/\IR$. 
\end{ex}

\section{Grothendieck groups}

In this subsection we will return to the general case and consider non-numerical stability conditions as well. To maintain continuity we will keep the notation ${\rm{Stab}}(Y)$ for the numerical stability manifold of a variety $Y$ and we will write ${\rm{Stab}}_*(Y)$ for the manifold of all locally finite stability conditions.

Theorem \ref{stab-mfd} tells us that the (numerical) stability manifold of a variety $Y$ is locally homeomorphic to a subspace of $\Hom(K(Y),\IC)$ (resp.\ $\Hom(N(Y),\IC)$). Thus, it is natural to ask what happens with the (numerical) stability manifold under scalar extension if we have an isomorphism $K(X_L)\cong K(X)$ (resp.\ $N(X_L)\cong N(X)$).

The next proposition gives a first answer under a slightly weaker assumption: 

\begin{prop}\label{isom-k-groups}
Assume that $L/K$ is a finite Galois extension and that the group homomorphism 
\begin{equation}
K(X)\otimes\IC \rightarrow K(X_L)\otimes \IC
\end{equation}
induced by $p^*$ is an isomorphism. Then ${\rm{Stab}}_*(X_L)_p={{\rm{Stab}}}_*(X_L)^G$. Similarly, if we have an isomorphism for the numerical Grothendieck groups, then ${\rm{Stab}}(X_L)_p={{\rm{Stab}}}(X_L)^G$.
\end{prop}

\begin{proof}
Let $\sigma=(Z,\kp) \in {\rm{Stab}}_*(X_L)_p$. We need to show that $Z$ is constant on the orbits of the action of the Galois group $G$. Note that our assumption gives an isomorphism 
\[\begin{xy} \xymatrix{\Hom(K(X), \IC) \ar[r]^\cong & \Hom(K(X_L),\IC).} \end{xy}\]
We therefore can write $Z=Z' \circ p_*$ for some $Z' \in \Hom(K(X), \IC)$. For an object $E \in \Db(X_L)$ and any $g \in G$ we then have
\[Z(g^*(E))=Z'p_*(g^*(E))=Z'p_*(E)=Z(E)\]
as claimed. The proof in the numerical case is similar.
\end{proof}

\begin{remark}
If (5.1) is an isomorphism, then it is easy to see that $\Stab(X)_\alpha=\Stab(X)_p$ and $\alpha=p_*$.
\end{remark}

Now recall from \cite[Ch.\ 6]{Bri1} that for any $\sigma=(Z, \kp)$ in the stability manifold of a variety $Y$ one can define a generalised norm on the vector space $\Hom(K(Y), \IC)$ by setting, for $U \in \Hom(K(Y), \IC)$:
\[\| U\|_\sigma=\sup \left\{ |U(E)|/|Z(E)| \; , \; E \; \rm{semistable \; in} \; \sigma \right\} \in [0, \infty].\]
In fact, for any connected component $\Sigma \subset {\rm{Stab}}(Y)$ the subspace $V(\Sigma)$ of Theorem \ref{stab-mfd} is the subspace of functions $U$ for which the norm is finite. Bridgeland proves furthermore that if $\sigma$ and $\tau$ are in the same connected component $\Sigma$, then the norms defined by these stability conditions are equivalent on $V(\Sigma)$.

\begin{lem}
Assume that (5.1) is an isomorphism and consider $\sigma=(Z, \kp) \in {\rm{Stab}}_*(X_L)_p$ and $p_*(\sigma) \in {\rm{Stab}}_*(X)$. Then the map
\[\begin{xy} \xymatrix{\Hom(K(X), \IC) \ar[r]^\cong & \Hom(K(X_L),\IC)} \end{xy}\] 
induced by $p^*$ is continuous with respect to the topologies induced by $\sigma$ and $p_*(\sigma)$. The same assertion holds for 
\[\begin{xy} \xymatrix{\Hom(K(X_L), \IC) \ar[r]^\cong & \Hom(K(X),\IC),} \end{xy}\]
where we consider the topologies induced by some $\sigma' \in {\rm{Stab}}_*(X)$ and $p^*(\sigma')$. In particular, these maps are homeomorphisms.
\end{lem}

\begin{proof}
Let $V=U\circ p^* \in \Hom(K(X), \IC)$ and recall that $p_*(\sigma)=(Z\circ p^*, \kp')$. Then

\begin{align*}
\| V\|_{p_*(\sigma)}&=\sup \left\{ |V(F)|/|Zp^*(F)| \; , \; F \; {\rm{semistable \; in}} \; p_*(\sigma)\right\}=\\
&=\sup \left\{ |Up^*(F)|/|Zp^*(F)| \; , \; F \; {\rm{semistable \; in}} \; p_*(\sigma)\right\}\leq \|U\|_\sigma
\end{align*}

since $F$ by definition is $p_*(\sigma)$-semistable if $p^*(F)$ is $\sigma$-semistable. The proof for $p_*$ is similar.

Now by Proposition \ref{alphabeta}, Lemma \ref{betaalpha} and Proposition \ref{isom-k-groups} we know that $p^*p_*(\sigma)=(dZ,\kp)$ and $p_*p^*(\sigma')=(dZ',\kp')$. Therefore the above already shows that the maps are homeomorphisms.
\end{proof}

\begin{cor}
Consider an element $\sigma \in \Sigma_L \subset {\rm{Stab}}_*(X_L)_p$, where $\Sigma_L$ is some connected component of ${\rm{Stab}}_*(X_L)_p$ of dimension $k$, and assume that (5.1) is an isomorphism. Then $p_*(\sigma)$ lies in a connected component $\Sigma$ of dimension $k$. Similarly, the dimension remains the same under $p^*$. The same holds for numerical stability conditions. 
\end{cor}

\begin{proof}
Follows immediately from the fact that $k$ is the dimension of $V(\Sigma_L)=\left\{ U \; , \; \| U\|_\sigma < \infty \right\}$, the above computation and the fact that $p_*p^*(\sigma')$ respectively $p^*p_*(\sigma)$ are in the same connected component as $\sigma$ resp.\ $\sigma'$ by Lemma \ref{alphabeta} resp.\ Lemma \ref{betaalpha}. The same proof works in the numerical case.
\end{proof}
 
\begin{prop}
Let (5.1) be an isomorphism, $\Sigma$ be a component in ${\rm{Stab}}_*(X)$ and $\Sigma_L$ the component of the same dimension in ${\rm{Stab}}_*(X_L)$ containing $p^*(\Sigma)$. Then we have $p^*(\Sigma)=\Sigma_L$. The same assertions hold for numerical stability conditions.
\end{prop}

\begin{proof}
Recall that we have an isomorphism between $V(\Sigma)$ and $V(\Sigma_L)$. Let $\sigma=(Z, \kp)=p^*(\sigma')=p^*(Z', \kp')$ be an element in $p^*(\Sigma) \subset \Sigma_L$, $U \subset \Sigma_L$ be an open neighbourhood of $\sigma$ homeomorphic to $U'\subset V(\Sigma_L)$ and $V$ be an open neighbourhood of $\sigma'$ homeomorphic to $V'\subset V(\Sigma)$. Restricting to the intersection $p_*(V') \cap U'$ if necessary and abusing notation we have a commutative diagram
\[\begin{xy}
\xymatrix{ V \ar[r]^{p^*} \ar[d]^\cong & U \ar[d]^\cong \\
					V(\Sigma) \ar[r]^\cong & V(\Sigma_L).}
\end{xy}\]
This shows that $p^*(\Sigma) \subset \Sigma_L$ is open. Since it is also closed by Corollary \ref{closed-submf} and $\Sigma_L$ is connected, we have the claimed equality. The proof for numerical stability conditions is similar.
\end{proof}

\begin{cor}\label{stab-iso}
If $N(X)\otimes \IC \cong N(X_L)\otimes \IC$, ${\rm{Stab}}(X_L)$ is connected and ${\rm{Stab}}(X)$ is non-empty, then ${\rm{Stab}}(X)\cong {\rm{Stab}}(X_L)$. A similar statement holds for the manifolds of all locally finite stability conditions.\qqed
\end{cor}

We will now apply the above results to the case of K3 surfaces. Recall that in \cite{Bri2} Bridgeland proved that the numerical stability manifold of a complex projective K3 surface $X$ contains a connected component which is a covering space of a certain open subset in the vector space $N(X)\otimes \IC$. Furthermore, he described the group of deck transformations of this covering via a subgroup of the group of autoequivalences $\Aut(\Db(X))$.\smallskip

Assume now that the K3 surface $S_\IC$ is defined over the real numbers and that $S$ possesses an $\IR$-rational point. Further assume that all line bundles on $S_\IC$ are also defined over $\IR$. This is, for example, the case if $S_\IC$ is generic, i.e.\ of Picard rank 1, since in this case the Galois group has to act as the identity on $\IZ=\Pic(S_\IC)$. These conditions ensure that (5.1) (or rather its numerical version) is an isomorphism. \smallskip

There cannot exist any numerical stability conditions on the standard heart ${\rm{Coh}}(S_\IC)$. Bridgeland therefore uses the theory of tilting to produce new hearts on which stability conditions can indeed be constructed. The method works as follows. One takes $\IR$-divisors $\beta$ and $\omega$ so that $\omega$ is in the ample cone. Any torsion free sheaf has a HN-filtration with respect to the slope-stability given by $\omega$ and truncating the filtrations at $\beta \cdot \omega$ gives a torsion pair. It follows from this construction that the torsion pair does not depend on $\beta$, but only on $\omega$ and the product $\beta\cdot \omega$. If $\omega$ is an ample line bundle (and therefore by our assumption defined over $\IR$), the HN-filtration of a sheaf pulled back from $S$, is defined over $\IR$, cf. \cite[Thm.\ 1.3.7]{HL}. The torsion pair therefore descends and so does the heart obtained by tilting with respect to it. Corollary \ref{heart-descent} then implies that the stability manifold ${\rm{Stab}}(S)$ is non-empty. Ignoring other possible components and using the corollary above we thus see that the distinguished component described by Bridgeland is defined over $\IR$. We thus proved the following

\begin{prop}
Let $S$ be a K3 surface over $\IR$ and denote by $S_\IC$ the complex K3 surface obtained by base change. Furthermore, assume that $S$ has an $\IR$-rational point and that $\Pic(S)=\Pic(S_\IC)$. Then there exists a connected component $\rm{Stab}^\dagger(S) \subset \rm{Stab}(S)$ such that there is a homeomorphism between $\rm{Stab}^\dagger(S)$ and Bridgeland's distinguished component $\rm{Stab}^\dagger(S_\IC) \subset \rm{Stab}(S_\IC)$. \qqed
\end{prop}

\end{document}